\documentclass[12pt,leqno,twoside]{article}
\usepackage{amssymb}
\usepackage{amsmath}
\usepackage{amsthm}
\usepackage{t1enc}
\usepackage[cp1250]{inputenc}
\usepackage{a4,indentfirst,latexsym}
\usepackage{graphics}
\usepackage{mathrsfs}
\usepackage{cite,enumitem,graphicx}
\usepackage[colorlinks=true,urlcolor=black,
citecolor=black,linkcolor=black,linktocpage,pdfpagelabels,
bookmarksnumbered,bookmarksopen]{hyperref}

\usepackage[english]{babel}

\input xy
\xyoption{all}

\newtheorem{Th}{Theorem}[section]

\newtheorem{Lem}[Th]{Lemma}
\newtheorem{Cor}[Th]{Corollary}

\newcommand{\cT}{{\mathcal T}}
\newcommand{\cTto}{\stackrel{\cT}{\longrightarrow}}

   \newcommand{\eps}{\varepsilon}

   \def\id{\mathrm{id}}

   \def\Z{\mathbb{Z}}

   \def\N{\mathbb{N}}
   \def\R{\mathbb{R}}

   \def\F{{\cal F}}

   \def\J{\mathcal{J}}

   \def\D{\mathcal{D}}

\begin{document}
\thispagestyle{empty}
\begin{center}
{\Large{\sc Solutions to a nonlinear Schr\"odinger equation with periodic potential and zero on the boundary of the spectrum}}

\vspace{5mm}

Jarosław Mederski\footnote{The study was supported by research grant NCN 2013/09/B/ST1/01963}
\\\mbox{}\\ 
\today
\end{center}

\begin{abstract}
We study the following nonlinear Schr\"odinger equation
$$
\left\{
\begin{array}{ll}
    -\Delta u + V(x) u = g(x,u)
    &
    \hbox{for } x\in\R^N,\\
    u(x)\to 0
    &
    \hbox{as } |x|\to\infty
\end{array}
\right.
$$
where $V:\R^N\to\R$ and $g:\R^N\times\R\to\R$ are periodic in $x$. We assume that $0$ is a right boundary point of the essential spectrum of $-\Delta+V$.  The superlinear and subcritical term g satisfies a Nehari type monotonicity condition. We employ a Nehari manifold type technique in a strongly indefitnite setting and obtain the existence of a ground state solution.  Moreover we get infinitely many geometrically distinct solutions provided that $g$ is odd.
\end{abstract}



\section*{Introduction}
\setcounter{section}{1}

We are concerned with the following nonlinear Schr\"odinger equation
\begin{equation}\label{NSE}
\left\{
\begin{array}{ll}
    -\Delta u + V(x) u = g(x,u)
    &
    \hbox{for } x\in\R^N,\\
    u(x)\to 0
    &
    \hbox{as } |x|\to\infty,
\end{array}
\right.
\end{equation}
where $V:\R^N\to\R$ is a periodic potential and $g:\R^N\times\R\to\R$ has superlinear growth. This equation appears in mathematical physics, e.g. when one studies standing waves $\Phi(x,t)=u(x)e^{-\frac{iEt}{\hslash}}$ of the time-dependent Schr\"odinger equation of the form
$$i\hslash\frac{\partial \Phi}{\partial t}=
-\frac{\hslash^2}{2m}\Delta \Phi+W(x)\Phi-f(x,|\Phi|)\Phi.$$
If the potential $V$ is periodic, then (\ref{NSE}) is of particular interest since it has a wide range of physical applications, e.g. in photonic crystals, where one considers periodic optical nanostructures (see \cite{Pankov} and references therein). It is well-known that  the spectrum $\sigma(-\Delta+V)$ of $-\Delta+V$ is purely continuous and may contain gaps, i.e. open intervals free of spectrum (see \cite{ReedSimon}). When $\inf\sigma(-\Delta+V)>0$ or $0$ lies in a gap of the spectrum $\sigma(-\Delta+V)$ 
then nonlinear Schr\"odinger equations have been widely investigated by many authors (see \cite{CotiZelati,Rabinowitz,AlamaLi,BuffoniJeanStuart,TroestlerWillem,KryszSzulkin,DingLee} and references therein) and nontrivial solutions to (\ref{NSE}) have been obtained. Ground state solutions, i.e. nontrivial solutions with the least possible energy, play an important role in physics and their existence has been studied e.g. in \cite{LiWangZeng,Pankov,SzulkinWeth,Liu}. If $V=0$ then $\sigma(-\Delta+V)=[0,+\infty)$ and the problem has been investigated in a classical work \cite{BerLions} or in a recent one \cite{AlvesSoutoMontenegro} (see also references therein). If $V$ is constant and negative then $0$ is an interior point of $\sigma(-\Delta+V)$ and solutions to (\ref{NSE}) have been found in \cite{EvequozWeth}.
\\ 
\indent In the present work, we focus on the situation when $0$ lies in the spectrum of $-\Delta+V$ and is the left endpoint of a spectral gap.
As far as we know there are only three papers dealing with this case. In \cite{BartschDingPeriodic} Bartsch and Ding obtained a nontrivial solution to (\ref{NSE}) assuming, among others, the following Ambrosetti-Rabinowitz condition:
\begin{equation}\label{ARcond}
g(x,u)u\geq \gamma G(x, u)  > 0\hbox{ for some }\gamma > 2\hbox{ and all }u \in\R\setminus\{0\},\; x\in\R^N,
\end{equation}
and a lower bound estimate:
\begin{equation}\label{LBcond}
G(x,u)\geq b |u|^{\mu}\hbox{ for some }b>0,\;\mu > 2\hbox{ and all }u \in\R,\; x\in\R^N,
\end{equation}
where $G$ is the primitive of $g$ with respect to $u$.
Applying a generalized linking theorem due to Kryszewski and Szulkin \cite{KryszSzulkin}, they proved that there is a solution in $H^2_{loc}(\R^N)\cap L^t(\R^N)$ for $\mu\leq t\leq 2^*$, where $2^*=\frac{2N}{N-2}$ if $N\geq 3$, and $2^*=\infty$ if $N=1,2$.
If $g$ is odd then the existence of infinitely many geometrically distinct solutions was obtained as well by means of an abstract critical point theory involving the $(PS)_I$-attractor concept (see Section 4 in \cite{BartschDingPeriodic} for details). In \cite{WillemZou} Willem and Zou relaxed condition (\ref{ARcond}) and they dealt with the lack of boundedness of Palais-Smale sequences. The authors developed the so-called monotonicity trick for strongly indefinite problems and established weak linking results. Recently Yang, Chen and Ding in \cite{YangChenDing} considered a Nehari-type monotone condition (see $(G5)$ below) instead of (\ref{ARcond}) and obtained a solution to (\ref{NSE}) using a variant of weak linking due to Schechter and Zou \cite{SchechterZou}. The lower bound estimate (\ref{LBcond}) has been assumed so far.
\\
\indent 
In this paper, our first aim is to prove the existence of a ground state solution to (\ref{NSE}) under the assumption that $0$ lies in the spectrum of $-\Delta+V$ and is the left endpoint of a spectral gap. As far as we know this is the first paper dealing with ground states in this case. Moreover, neither (\ref{ARcond}) nor (\ref{LBcond}) are assumed. 
Namely, throughout the paper we impose the following conditions.\\
 
\noindent {\bf (V)} $V\in C(\R^N,\R)$, $V$ is $1$-periodic in $x_i$, $i=1,2,...,N$,
$0\in\sigma(-\Delta+V)$ and there exists $\beta>0$ such that $(0,\beta]\cap\sigma(-\Delta+V)=\emptyset$.

\noindent {\bf (G1)} $g\in C(\R^N\times\R,\R)$, $g$ is $1$-periodic in $x_i$, $i=1,2,...,N.$

\noindent {\bf (G2)} There are $a>0$ and $2<\mu\leq p<2^*$ such that
$$|g(x,u)|\leq a(|u|^{\mu-1}+|u|^{p-1})\hbox{ for all }u \in\R,\; x\in\R^N.$$

\noindent {\bf (G3)} There is $b>0$ such that 
$$G(x,u)\geq b |u|^{\mu}\hbox{ for all }|u|\leq 1,\; x\in\R^N.$$

\noindent {\bf (G4)} $G(x,u)/|u|^2\to\infty$ uniformly in $x$ as $|u|\to\infty$.

\noindent {\bf (G5)} $u\mapsto g(x,u)/|u|$ is strictly increasing on $(-\infty,0)$ and $(0,\infty)$.

We point out that $(G3)$ and $(G4)$ are substantially weaker than (\ref{LBcond}). Indeed, take a nonlinearity of the type 
$$g(x,u)=q(x)u\ln(1+|u|^{p-2}),\; q(x)\geq \inf_{\R^N} q >0,$$
where $q:\R^N\to\R$ is continuous and $1$-periodic in $x_i$, $i=1,2,...,N$. Observe that conditions $(G1)-(G5)$ are obeyed with 
$2<\mu=p<2^*$,
but (\ref{LBcond}) does not hold. The above  nonlinearity has recently attracted attention of many authors since the Ambrosetti-Rabinowitz condition (\ref{ARcond}) is not satisfied and thus Palais-Smale sequences do not have to be bounded (see e.g. \cite{Jeanjean,LiuWang,DingLee,MiyagakiSouto}).\\
\indent Assumptions $(V)$, $(G1)-(G5)$ allow to find a function space $E_{2,\mu}$ (see Section \ref{SectionVar}) on which the energy functional associated to (\ref{NSE})
\begin{equation}\label{eqJFormula}
\J(u):=\frac{1}{2}\int_{\R^N}|\nabla u|^2+V(x)|u|^2\;dx-\int_{\R^N}G(x,u)\;dx
\end{equation}
is a well-defined $C^1$-map. Moreover critical points of $\J$ correspond to  solutions to (\ref{NSE}). In order to find ground state solutions we consider the Nehari-Pankov manifold $\mathcal{N}\subset E_{2,\mu}$ defined later by (\ref{NehariDef}).\\
\indent Our main results read as follows. For the precise definitions see the next sections.
\begin{Th}\label{Th1}
If assumptions $(V)$, $(G1)-(G5)$ are satisfied then $(\ref{NSE})$ has a ground state solution $u\in\mathcal{N}$ such that $\J(u)=\inf_{\mathcal{N}} \J>0$. Moreover $u\in H^2_{loc}(\R^N)\cap L^t(\R^N)$ for $\mu\leq t\leq 2^*$.
\end{Th}

Furthermore, we establish the following multiplicity result and we would like to emphasize that (\ref{ARcond}) is not assumed as opposed to \cite{BartschDingPeriodic}.

\begin{Th}\label{Th2}
If assumptions $(V)$, $(G1)-(G5)$ are satisfied, $g$ is odd in $u$, then $(\ref{NSE})$ has infinitely many pairs $\pm u$ of geometrically distinct solutions in $H^2_{loc}(\R^N)\cap L^t(\R^N)$ for $\mu\leq t\leq 2^*$.
\end{Th}

The paper is organized as follows. In the next section we formulate a variational approach to $(\ref{NSE})$ and
we define a function space $E_{2,\mu}$ such that the energy functional $\J:E_{2,\mu}\to\R$ associated with $(\ref{NSE})$ is a well-defined $C^1$-map. Moreover some embeddings results of $E_{2,\mu}$ are established. In Section \ref{SectionAbstactSetting} we recall the recently obtained critical point theory from \cite{BartschMederski} which allows to deal with the underlying geometry of $\J$. Next, in Section \ref{SectionGroundState}, we introduce the Nehari-Pankov manifold $\mathcal{N}\subset E_{2,\mu}$ on which we minimize $\J$ to find a ground state and we prove Theorem \ref{Th1}. Finally, in the last Section \ref{SectionMultiplicity}, the multiplicity result is obtained.

\section{Variational setting}\label{SectionVar}

Let $H^1(\R^N)$ denote the Sobolev space with the norm $\|\cdot\|_{H^1}$.
Let us consider a functional $\J:H^1(\R^N)\to\R$  given by formula (\ref{eqJFormula}). We note that $\J$ is of class $C^1$ and its critical points correspond to solutions to (\ref{NSE}). 
By assumption $(V)$, $H^1(\R^N)$ has the decomposition of the form $E^+\oplus E'$ corresponding to the decomposition of spectrum of $\sigma(S)$ into $\sigma(S)\cap[\beta,\infty)$ and
$\sigma(S)\cap(-\infty,0]$, where $S:=-\Delta+V$ with domain $\D(S)=H^2(\R^N)$. We can define a new norm $\|\cdot\|_E$ on $E^{+}$ (resp. $E'$) by setting
$$\|u^{+}\|_E^2:=\int_{\R^N}|\nabla u^+|^2+V(x)|u^+|^2\;dx$$
and 
$$\|u'\|^2_E:=-\int_{\R^N}|\nabla u'|^2+V(x)|u'|^2\;dx$$
for $u^+\in E^+$ and $u'\in E'$.
Then $\|\cdot \|_E$ is equivalent to $\|\cdot\|_{H^1}$ on $E^+$ and
is weaker than $\|\cdot\|_{H^1}$ on $E'$ (see \cite{BartschDingPeriodic}). Let $E$ be the completion of $H^1(\R^N)$ with respect to $\|\cdot\|_E$. Then $H^1(\R^N)=E^+\oplus E'$ is continuously embedded in $E$ and $E$ is a Hilbert space with the inner product
$\langle u,v\rangle_E:=\langle |S|^{\frac{1}{2}}u,|S|^{\frac{1}{2}}v\rangle_{L^2}$, where $\langle\cdot,\cdot\rangle_{L^2}$ is the usual inner product in $L^2(\R^N)$.
Note that $\J$ can be written as follows
$$\J(u)=\frac{1}{2}(\|u^+\|_E^2-\|u'\|_E^2)-\int_{\R^N}G(x,u)\;dx=\frac{1}{2}\|u^+\|^2_E-I(u),$$
where 
$$I(u):=\frac{1}{2}\|u'\|_E^2+\int_{\R^N}G(x,u)\;dx$$
for any $u=u^++u'\in E^+\oplus E'$. We do not know if $\J$ has critical points in $H^1(\R^N)$. Moreover $I$ is not defined on $E$ owing to our assumptions on $g(x,u)$. Therefore we are going to define a space $E_{2,\mu}$ such that there are continuous embeddings
$$H^1(\R^N)\subset E_{2,\mu}\subset E,$$ 
$I$ is well-defined on $E_{2,\mu}$ and $\J$ admits critical points on $E_{2,\mu}$.

\subsection{Function space}

Let $(P_{\lambda}:L^2(\R^N)\to L^2(\R^N))_{\lambda\in\R}$ denote the spectral family of $S$. Let $L':=P_0(L^2(\R^N))$ and $L^+:=(\id- P_0)(L^2(\R^N))$. Then we have the orthogonal decomposition $L^2(\R^N)=L^+\oplus L'$ and then $E^+=H^1(\R^N)\cap L^+$, $E'=H^1(\R^N)\cap L'$ (see \cite{ReedSimon,PankovNotes}). Moreover
$$\|u\|_E^2=\int_{-\infty}^\infty |\lambda|\; d|P_{\lambda}u|_2^2,$$
where here and in the sequel, $|\cdot|_k$ denotes the usual norm in $L^k(\R^N)$ for any $k\geq 1$.\\
\indent Let us assume that $2\leq \nu\leq\mu$. By
$L^{\nu,\mu}(\R^N):=L^\nu(\R^N)+L^\mu(\R^N)$ we denote the Banach space of all functions of the form $v=v_1+v_2$, where $v_1\in L^\nu(\R^N)$ and $v_2\in L^\mu(\R^N)$,
endowed with the following norm
$$|v|_{\nu,\mu}:=\inf\{|v_1|_{\nu}+|v_2|_\mu|\;v=v_1+v_2\}.$$
By \cite{BadialePisaniRolando}[Prop. 2.5] the infimum in $|\cdot|_{\nu,\mu}$ is attained. Moreover, there is a continuous embedding 
$$L^t(\R^N)\subset L^{\nu,\mu}(\R^N)$$
for any $\nu\leq t\leq\mu$ and, if $\nu=\mu$ then norms $|\cdot|_{\nu,\mu}$ and $|\cdot|_{\mu}$ are equivalent.
Let
$E_{\nu,\mu}'$ and $E_{\mu}'$ be the completions of $E'$ with respect to the norms
$$\|\cdot\|_{\nu,\mu}=(\|\cdot\|_E^2+|\cdot |_{\nu,\mu}^2)^{\frac{1}{2}},$$
and
$$\|\cdot\|_{\mu}=(\|\cdot\|_E^2+|\cdot |_{\mu}^2)^{\frac{1}{2}}$$
respectively.
Thus we have
the following continuous embeddings
$$E'\subset E'_{\mu}\subset E'_{\nu,\mu}\subset E.$$ 
Space $E'_{\mu}$ has been introduced in \cite{BartschDingPeriodic} and note that, if $\nu=\mu$ then $E'_{\nu,\mu}=E'_{\mu}$ and the norms $\|\cdot\|_{\nu,\mu}$ and $\|\cdot\|_{\mu}$ are equivalent. In our setting, space $E'_{\nu,\mu}$ with $\nu=2$ plays an important role because of superlinear growth conditions  $(G3)$ and $(G4)$ (cf. Lemma \ref{LemG3estimate}).
The following somewhat surprising observation is crucial for continuous embeddings of $E'_{\nu,\mu}$ into $L^t(\R^N)$ (see Lemma \ref{LemContEmbed}).

\begin{Lem}\label{LemIdentification}
$E'_{\nu,\mu}=E'_{\mu}$ and norms $\|\cdot\|_{\nu,\mu}$, $\|\cdot\|_{\mu}$ are equivalent for any $2\leq \nu\leq \mu \leq 2^*$.
\end{Lem}
\begin{proof}
Note that it is enough to show the inclusion $E_{\nu,\mu}'\subset E_{\mu}'$.
Let $u\in E_{\nu,\mu}'$ and we proceed as follows.\\
{\em Step 1.} For any $y\in\R^N$, $r>0$ and $\eps>0$ we have $u\in H^2(B(y,r))$ and
\begin{equation}\label{eqCalderonZygmund}
\|u\|_{H^2(B(y,r))}\leq c(|u|_{L^2(B(y,r+\eps))}+|Su|_{L^2(B(y,r+\eps))})
\end{equation}
for some constant $c>0$ depending on $r$ and $\eps$.\\
Indeed, similarly as in proof of \cite{BartschDingPeriodic}[Lem 2.1], take a sequence $(u_n)_{n\in \N}\subset E'$ such that $\|u_n-u\|_{\nu,\mu}\to 0$ as $n\to\infty$.  
Note that $E'\subset L'\subset \D(S)=H^2(\R^N)$ because the spectrum of $S$ is bounded below. Since
\begin{eqnarray*}
|S(u_n-u_m)|_2^2&=&\int_{-\infty}^0\lambda^2 \;d |P_{\lambda}(u_n-u_m)|^2_2\\
&\leq& \alpha \int_{\alpha}^0\lambda \;d |P_{\lambda}(u_n-u_m)|_2^2\\
&=&-\alpha \|u_n-u_m\|_E^2,
\end{eqnarray*}
where $\alpha<\inf\sigma(S)<0$,
then  $Su_n$ is a Cauchy sequence in $L^2(\R^N)$. Since  $u_n\to u$ in $L^{\nu,\mu}(\R^N)$ then, by \cite{BadialePisaniRolando}[Prop. 2.14], $u_n\to u$ in
$L^{\nu}(\Omega)$ for any bounded and measurable $\Omega\subset\R^N$, hence the convergence holds in $L^{2}(\Omega)$ as well.
In view of the Calderon-Zygmund inequality (see \cite{GilbargTrudinger}[Th. 9.11]) there is a constant $c>0$ such that
$$\|u_n-u_m\|_{H^2(B(y,r))}\leq c(|u_n-u_m|_{L^2(B(y,r+\eps))}+|S(u_n-u_m)|_{L^2(B(y,r+\eps))}).$$
Thus $u\in H^2(B(y,r))$ and again by the Calderon-Zygmund inequality (\ref{eqCalderonZygmund}) holds.\\
{\em Step 2.}  
\begin{equation}\label{eqUinL2star}
u\in L^{2^*}(\R^N).
\end{equation}
In view of \cite{BadialePisaniRolando}[Prop. 2.5]
there are $u_1\in L^{\nu}(\R^N)$ and $u_2\in L^{\mu}(\R^N)$
such that $u=u_1+u_2$ and
\begin{equation*}\label{eqNormDecomp}
|u|_{\nu,\mu,2^*}^{2^*}=|u_1|_{\nu}^{2^*}+|u_2|_{\mu}^{2^*}
\end{equation*}
where 
$$|v|_{\nu,\mu,k}=(\inf\{|v_1|_{\nu}^k+|v_2|_\mu^k|\;v=v_1+v_2,\;v_1\in L^{\nu}(\R^N),\;v_2\in L^{\mu}(\R^N)\})^{\frac{1}{k}}$$
defines a family of equivalent norms on $L^{\nu,\mu}(\R^N)$ for $k\geq 1$ (see also \cite{BadialePisaniRolando}[Prop. 2.4]).
Observe that from (\ref{eqCalderonZygmund}),  for any $y\in\R^N$, $r>0$ and $\eps>0$
\begin{eqnarray*}
|u|_{L^{2^*}(B(y,r))}&\leq& c_1(|u|_{L^{2}(B(y,r+\eps))}+|Su|_{L^2(B(y,r+\eps))})\\
&\leq& c_1(|u_1|_{L^{2}(B(y,r+\eps))}+|u_2|_{L^{2}(B(y,r+\eps))}+|Su|_{L^2(B(y,r+\eps))})\\
&\leq& c_2(|u_1|_{L^{\nu}(B(y,r+\eps))}+|u_2|_{L^{\mu}(B(y,r+\eps))}+|Su|_{L^2(B(y,r+\eps))}),
\end{eqnarray*}
for some constants $c_1,c_2>0$ depending on $r$ and $\eps$.
Therefore
\begin{eqnarray*}
\int_{B(y,r)}|u|^{2^*}\;dx&\leq&
c_3\Big(|u_1|^{2^*-\nu}_{\nu}
\int_{B(y,r+\eps)}|u_1|^{\nu}\;dx
+|u_2|^{2^*-\mu}_{\mu}
\int_{B(y,r+\eps)}|u_2|^{\mu}\;dx\\
&&+|Su|^{2^*-2}_{2}
\int_{B(y,r+\eps)}|Su|^{2}\;dx\Big) 
\end{eqnarray*}
for some constant $c_3>0$.
For any $r>0$ there is $\eps>0$ and a covering of $\R^N$ by balls $\{B(y,r)\}_{y\in Y}$, where $Y\subset \R^N$ such that each point of $\R^N$ is contained in at most $N+1$ balls $B(y,r+\eps)$. Therefore 
\begin{eqnarray*}
\int_{\R^N}|u|^{2^*}\;dx&\leq&
(N+1)c_3\Big(|u_1|^{2^*}_{\nu}
+|u_2|^{2^*}_{\mu}
+|Su|^{2^*}_{2}\Big)\\
&=&
(N+1)c_3\Big(|u|^{2^*}_{\nu,\mu,2^*}
+|Su|^{2^*}_{2}\Big).
\end{eqnarray*}
Since norms $|\cdot|_{\nu,\mu,2^*}$ and $|\cdot|_{\nu,\mu,1}=|\cdot|_{\nu,\mu}$ are equivalent, then $u\in L^{2^*}(\R^N)$.\\
{\em Step 3.} $u\in E_{\mu}(\R^N)$.\\
Indeed, since $u\in L^{\nu,\mu}(\R^N)$ then by \cite{BadialePisaniRolando}[Prop. 2.3] we obtain
$u\in L^{\nu}(\Omega_u)\cap L^{\mu}(\Omega_u^c)$, where 
\begin{equation*}
\Omega_u:=\{x\in\R^N|\; |u(x)|>1\}
\end{equation*}
has finite Lebesgue measure. Since $u\in L^{2^*}(\Omega_u)$ then by the interpolation inequality we get $u\in L^{\mu}(\Omega_u)$. Hence $u\in L^{\mu}(\R^N)$ and $u\in E_{\mu}(\R^N)$.
\end{proof}

From (\ref{eqCalderonZygmund}) and (\ref{eqUinL2star}) or by \cite{BartschDingPeriodic}[Lem. 2.1] we infer the following embeddings.

\begin{Lem}\label{LemContEmbed}
If $2\leq \nu\leq \mu\leq 2^*$ then $E_{\nu,\mu}'$ embeds continuously into $H^2_{loc}(\R^N)$ and  $L^t(\R^N)$ for $\mu\leq t\leq 2^*$,
and compactly into $L^t_{loc}(\R^N)$ 
for $2\leq t< 2^*$.
\end{Lem}

Observe that we obtain continuous embeddings
$$H^1(\R^N)\subset E_{\nu,\mu}:=E^+\oplus E'_{\nu,\mu}\subset E$$ 
where $E_{\nu,\mu}$ is endowed with the norm
$$\|u\|:=(\|u^+\|^2_E+\|u'\|_{\nu,\mu}^2)^{\frac{1}{2}}$$
for $u=u^{+}+u'\in E^+\oplus E'_{\nu,\mu}$. Since $|\cdot|_{\nu,\mu}$ is uniformly convex (see \cite{BadialePisaniRolando}[Prop. 2.6]), then $E_{\nu,\mu}$ is reflexive and bounded sequences in 
$E_{\nu,\mu}$ are relatively weakly compact. In view of the Sobolev embeddings, Lemma \ref{LemContEmbed} holds also for $E_{\nu,\mu}$ and $\J:E_{\nu,\mu}\to\R$ given by (\ref{eqJFormula}) is a well-defined $C^1$-map.
Moreover from Lemma \ref{LemIdentification} and \cite{BartschDingPeriodic}[Cor. 2.3] we get that a solution to (\ref{NSE}) in $E_{\nu,\mu}$ vanishes at infinity.

\begin{Cor}\label{CorollaryUVanish}
If $u\in E_{\nu,\mu}$ solves $(\ref{NSE})$ then $u(x)\to 0$ as $|x|\to\infty$.
\end{Cor}

\section{Abstract setting}\label{SectionAbstactSetting}

In this section we are going 
recall the recent abstract result obtained in \cite{BartschMederski} which seems to be
appropriate in dealing with the geometry and the regularity of energy functional $\J$.\\
\indent For the purpose of this section we assume that $X$ is an arbitrary reflexive Banach space with norm $\|\cdot\|$ such that
$X=X^{+}\oplus X'$, $X^+$, $X'$ are closed subspaces of $X$ and $X^+\cap X'=\{0\}$. If $u\in X$ then there is the unique decomposition $u=u^++u'$ where $u^+\in X^{+}$ and $u'\in X'$. We may also assume that $\|u\|^2=\|u^+\|^2+\|u'\|^2$. 
In order to ensure that a unit sphere in $X^+$ 
$$S^+:=\{u\in X^+|\;\|u\|=1\}$$
is a $C^1$-submanifold of $X^+$, we assume that $X^+$ is a Hilbert space with the scalar product $\langle\cdot,\cdot\rangle$ such that $\langle u,u\rangle = \|u\|^2$ for any $u\in X^+$. In addition to the norm topology we need the topology $\cT$ on $X$ which is the product of the norm topology in $X^+$ and the weak topology in $X'$. In particular, $u_n\cTto u$ provided that $u_n^+ \to u^+$ and $u_n' \rightharpoonup u'$.\\
We define the following Nehari-Pankov manifold (cf. \cite{Pankov})
$$\mathcal{N}:=\{u\in X\setminus X'|\; \J'(u)(u)=0,\;\J'(u)(h')=0 \hbox{ for any } h'\in X'\}.$$
We say that $\J$ satisfies the $(PS)_c^\cT$-condition in $\mathcal{N}$ if every $(PS)_c$-sequence in $\mathcal{N}$ has a subsequence which converges in $\cT$:
\[
u_n \in \mathcal{N},\ \J'(u_n) \to 0,\ \J(u_n) \to c \qquad\Longrightarrow\qquad
u_n \cTto u\in X\ \text{ along a subsequence}.
\]

\begin{Th}[see \cite{BartschMederski}]\label{ThLink1}
Let $J\in C^1(X,\R)$ be a map of the form
\begin{equation}\label{EqJ}
\J(u)=\frac{1}{2}\|u^+\|^2-I(u)
\end{equation}
for any $u=u^++u'\in X^{+}\oplus X'$ such that:\\
$(J1)$ $I(u)\geq I(0)=0$ for any $u\in X$ and, $I$ is $\cT$-sequentially lower semicontinuous, i.e. if $u_n\cTto u_0$ then $\liminf_{n\to\infty}I(u_n)\geq I(u_0)$.\\
$(J2)$ If $u_n\cTto u_0$ and $I(u_n)\to I(u_0)$ then
$u_n\to u_0$.\\
$(J3)$ If $u\in \mathcal{N}$ then $\J(u)>\J(tu+h')$ for any $t\geq 0$, $h'\in X'$ such that $tu+h'\neq u$.\\
$(J4)$ $0<\inf_{u\in X^+,\;\|u\|=r}\J(u)$.\\
$(J5)$ $\|u^+\|+I(u)\to\infty$ as $\|u\|\to\infty$.\\
$(J6)$ $I(t_nu_n)/t_n^2\to\infty$ if $t_n\to\infty$  and $u_n^+\to u_0^+$ for some $u_0^+\neq 0$ as $n\to\infty$.\\
Then:\\
$(a)$ $c:=\inf_{\mathcal{N}}\J>0$ and there exists a $(PS)_c$-sequence $(u_n)_{n\in\N}\subset\mathcal{N}$, i.e. $\J(u_n)\to c$ and $\J'(u_n)\to0$ as $n\to\infty$. If $\J$ satisfies the $(PS)_{c}^\cT$-condition in $\mathcal{N}$ then $c$ is achieved by a critical point of $\J$.\\
$(b)$ There is a homeomorphism $n:S^+\to \mathcal{N}$ such that $n^{-1}(u)=\frac{u^+}{\|u^+\|}$, $n(u)$ is the unique maximum of $\J$ on $\R^+u\oplus X'$ for $u\in\mathcal{N}$ and $\J\circ n:S^+\to\R$ is of class $C^1$. Moreover, a sequence $(u_n)\subset S^+$ is a Palais-Smale sequence for $\J\circ n$ if and only if $n(u_n)\subset\mathcal{N}$ is a Palais-Smale sequence for $\J$, and $u\in S^+$ is a critical point of $\J\circ n$ if and only if $n(u)$ is a critical point of $\J$.
\end{Th}

Proof of Theorem \ref{ThLink1} is based on the Ekeland's variational applied to a map $\J\circ n:S^+\to\R$. Some steps of the proof are enlisted in $(b)$ since they play a crucial role in Section \ref{SectionMultiplicity} (see \cite{BartschMederski}, cf. \cite{SzulkinWethHandbook}).

\section{Ground state solutions}\label{SectionGroundState}

We are going to look for critical points of $\J:E_{2,\mu}\to\R$ on the following {\em Nehari-Pankov manifold}
\begin{equation}\label{NehariDef}
\mathcal{N}:=\{u\in E_{2,\mu}\setminus E'_{2,\mu}|\; \J'(u)(u)=0,\;\J'(u)(h')=0 \hbox{ for any }h'\in E'_{2,\mu}\}.
\end{equation}
The idea to consider a Nehari-type manifold for indefinite problems was firstly observed by Pankov in \cite{Pankov}. If $\J\in C^2(E_{2,\mu},\R)$ and under some additional assumptions, $\mathcal{N}$ is a $C^1$-submanifold of $E_{2,\mu}$ (see \cite{Pankov,SzulkinWethHandbook}). However we assume only that $\J$ is of $C^1$-class and $\mathcal{N}$ does not have to be a $C^1$-submanifold of $E_{2,\mu}$.
In order to find a minimizing Palais-Smale sequence we need to check assumptions  $(J1)$ - $(J6)$ of Theorem \ref{ThLink1} by setting $X^+:=E^+$ and $X':=E_{2,\mu}'$. Firstly observe that the following inequality holds.

\begin{Lem}\label{LemG3estimate}
There is a constant $c>0$ such that
for any $u\in E_{2,\mu}$
\begin{equation}\label{eqG3estimate}
\int_{\R^N}G(x,u)\;dx\geq
c\min\{|u|_{2,\mu}^2,
|u|_{2,\mu}^\mu\}.
\end{equation}
\end{Lem}
\begin{proof}
Note that by $(G2)$ and $(G5)$ we know that $G(x,u)>0$ if  $u\neq 0$. Therefore $(G3)$ and $(G4)$ imply that
there is $b'>0$ such that
\begin{equation}
G(x,u)\geq b'\min\{|u|^{2},|u|^{\mu}\}
\end{equation}
for all $u\in\R$ and $x\in\R^N$.
Then we infer that for $u\in E_{2,\mu}$
\begin{eqnarray*}
\int_{\R^N}G(x,u)\;dx&\geq& b' \Big( \int_{\Omega_u}|u|^2\;dx+
\int_{\Omega_u^c}|u|^\mu\;dx\Big)\\
&=& b'(|u\chi_{\Omega_u}|_2^2 +|u\chi_{\Omega^c_u}|_\mu^\mu),
\end{eqnarray*}
where $\chi$ denoted the characteristic function and $$\Omega_u:=\{x\in\R^N|\; |u(x)|>1\}$$
is bounded.
In view of \cite{BadialePisaniRolando}[Prop. 2.4])
$$|u|_{2,\mu,\infty}:=\inf\{\max\{|u_1|_{2},|u_2|_\mu\}|\;u=u_1+u_2,\;u_1\in L^{2}(\R^N),\;u_2\in L^{\mu}(\R^N)\}$$
defines a norm on $L^{2,\mu}(\R^N)$ equivalent with $|\cdot|_{2,\mu}$.
Observe that  if $|u\chi_{\Omega_u}|_2\geq |u\chi_{\Omega_u^c}|_\mu$
then
$$|u\chi_{\Omega_u}|_2^2 +|u\chi_{\Omega^c_u}|_\mu^\mu\geq 
(\max\{|u\chi_{\Omega_u}|_2,|u\chi_{\Omega^c_u}|_\mu\})^2\geq |u|_{2,\mu,\infty}^2$$
and if $|u\chi_{\Omega_u}|_2< |u\chi_{\Omega_u^c}|_\mu$
then
$$|u\chi_{\Omega_u}|_2^2 +|u\chi_{\Omega^c_u}|_\mu^\mu\geq 
(\max\{|u\chi_{\Omega_u}|_2,|u\chi_{\Omega^c_u}|_\mu\})^\mu\geq |u|_{2,\mu,\infty}^\mu.$$
Therefore
\begin{eqnarray*}
\int_{\R^N}G(x,u)\;dx&\geq&  b' \min\{|u|_{2,\mu,\infty}^2,
|u|_{2,\mu,\infty}^\mu\}\nonumber\\
&\geq& c\min\{|u|_{2,\mu}^2,
|u|_{2,\mu}^\mu\},\nonumber
\end{eqnarray*}
for some constant $c>0$.
\end{proof}

The following lemma shows that $(J4)$ - $(J6)$ hold for $\J$.

\begin{Lem}\label{LemLGcondV}
The following conditions hold.\\
$(a)$ $0<\inf_{u\in E^+,\;\|u\|=r}\J(u)$.\\
$(b)$ $\|u^+\|+I(u)\to\infty$ as $\|u\|\to\infty$.\\
$(c)$ $I(t_nu_n)/t_n^2\to\infty$ if $u_n^+\to u_0^+$ for some $u^+_0\neq 0$ and $t_n\to\infty$ as $n\to\infty$.
\end{Lem}
\begin{proof}
$(a)$ If $u\in E^{+}$ then by $(G2)$
$$\J(u)\geq\frac{1}{2}\|u\|_E^2-\frac{a}{\mu}|u|^\mu_\mu
-\frac{a}{p}|u|^p_p.$$
Since $E^{+}$ is continuously embedded in $L^{\mu}(\R^N)$ and in $L^{p}(\R^N)$ then 
$$\J(u)\geq\frac{1}{2}\|u\|_E^2-C_1(\|u\|_E^{\mu}+\|u\|_E^p)$$
for some constant $C_1>0$.
Thus we get the inequality in $(a)$.\\
$(b)$ Suppose that $\|u_n\|\to\infty $ as $n\to\infty$ and $(\|u_n^{+}\|_E)_{n\in\N}$ is bounded. Then $(|u_n^{+}|_{2,\mu})_{n\in\N}$ is bounded and
$$\|u_n'\|^2_{2,\mu}=\|u_n'\|^2_{E}+|u_n'|^2_{2,\mu}\to\infty$$
as $n\to\infty$. If along a subsequence $\|u_n'\|_{E}\to\infty$ then obviously $I(u_n)\to\infty$. Assume that $(\|u_n'\|_E)_{n\in\N}$ is bounded.
Then $|u_n'|_{2,\mu}\to\infty$ and by (\ref{eqG3estimate}),
$I(u_n)\to\infty$ as $n\to\infty$.\\
$(c)$ Suppose that, up to a subsequence, $I(t_n u_n)/t_n^2$ is bounded, $u_n^+\to u_0^+$ for some $u_0^+\in E^{+}\setminus\{0\}$ and $t_n\to\infty$ as $n\to\infty$. Note that by (\ref{eqG3estimate})
\begin{equation*}
\frac{I(t_nu_n)}{t_n^2}\geq\frac{1}{2}\|u_n'\|^2_E+c\min\{|u_n|_{2,\mu}^2,t_n^{\mu-2}|u_n|_{2,\mu}^\mu\},
\end{equation*}
and then  $(\|u_n'\|_{2,\mu})_{n\in\N}$ is bounded. In view of Lemma \ref{LemContEmbed} we may assume that $u_n'\rightharpoonup u_0'$ in $E_{2,\mu}'$ and $u_n'(x)\to u'_0(x)$ a.e. on $\R^N$.
If the Lebesgue measure $|\Omega|>0$, where
$\Omega:=\{x\in\R^N|\;u^+_0(x)+u'_0(x)\neq 0\}$, then by $(G4)$ and Fatou's lemma
$$\int_{\R^N}\frac{G(x,t_nu_n)}{t_n^2}\;dx\to\infty.$$
Thus we obtain that $I(t_n u_n)/t_n^2\to\infty$ which is a contradiction. Therefore $|\Omega|=0$ and 
$u_0'=-u_0^+$ a.e. on $\R^N$.
Since $\langle u_0',u_0^+\rangle_E=0$ then $u_0^+=0$. The obtained contradiction implies that
$I(t_nu_n)/t_n^2\to\infty$.
\end{proof}

We recall that $u_n\cTto u_0$ provided that $u_n^+\to u_0^+$ in $E^+$ and $u_n'\rightharpoonup u_0'$ in $E_{2,\mu}'$ (see Section \ref{SectionAbstactSetting}).

\begin{Lem}\label{LemLG}
The following conditions hold.\\
$(a)$ $I(u)\geq 0$ for any $u\in E_{2,\mu}$ and $I$ is $\cT$-sequentially lower semicontinuous.\\
$(b)$ If $u_n\cTto u_0$ and $I(u_n)\to I(u_0)$ then
$u_n\to u_0$.\\
$(c)$ If $u\in \mathcal{N}$ then $\J(u)>\J(tu+h')$ for any $t\geq 0$, $h'\in E'_{2,\mu}$ such that $tu+h'\neq u$.
\end{Lem}
\begin{proof}
$(a)$ Let $u_n\cTto u_0$. Since $E_{2,\mu}$ is compactly embedded in $L^{2}_{loc}(\R^N)$, then we may assume that $u_n\to u_0$ in $L^{2}_{loc}(\R^N)$ and $u_n(x)\to u(x)$ a.e. in $\R^N$. In view of the Fatou's lemma and the weakly sequentially lower semicontinuity of the map $E' \ni u'\mapsto \frac{1}{2}\|u'\|^2_E\in\R$, we get $\liminf_{n\to\infty}I(u_n)\geq I(u_0)$.\\
$(b)$  Let $u_n\cTto u_0$ and $I(u_n)\to I(u_0)$. Since
$E' \ni u'\mapsto \frac{1}{2}\|u'\|^2_E\in\R$ is weakly sequentially lower semicontinuous and
$E_{2,\mu} \ni u\mapsto \int_{\R^N}G(x,u)\;dx\in\R$ is
$\cT$-sequentially lower semicontinuous, then 
$$\lim_{n\to\infty}\|u'_n\|_E^2=\|u_0'\|_E^2$$ 
and 
\begin{equation}\label{EqConvG}
\lim_{n\to\infty} \int_{\R^N}G(x,u_n)\;dx= \int_{\R^N}G(x,u_0)\;dx. 
\end{equation}
Note that, along a subsequence,
$$\|u_n'-u_0'\|^2_E=\|u_n'\|_E^2-\|u_0'\|^2_E-2\langle u_n'-u_0',u_0'\rangle_{E}\to 0.$$
Hence $u_n=u_n^++u_n'\to u_0=u_0^++u_0'$ in $E$. Thus we need to show that $u_n'\to u_0'$ in $L^{2,\mu}(\R^N)$.
Since $E_{2,\mu}$ is compactly embedded in $L_{loc}^2(\R^N)$, then we may assume that $u_n(x)\to u_0(x)$ a.e. in $\R^N$.
Observe that
\begin{eqnarray}\label{EqBrezisLieb}
\int_{\R^N}G(x,u_n)-G(x,u_n-u_0)\; dx
&=&\int_{\R^N}\int_0^1\frac{d}{dt}G(x,u_n-u_0+tu_0)\; dtdx\\\nonumber
&=&\int_0^1\int_{\R^N}g(x,u_n-u_0+tu_0)u_0\; dxdt.
\end{eqnarray}
Thus by $(G2)$ for any $\Omega\subset\R^N$
\begin{eqnarray*}
\int_{\Omega}|g(x,u_n-u_0+tu_0)u_0|\; dx
&\leq& a|u_n-u_0+tu_0|_{\mu}^{\mu-1}|u_0\chi_{\Omega}|_{\mu}\\
&&+a|u_n-u_0+tu_0|_{p}^{p-1}|u_0\chi_{\Omega}|_{p}
\end{eqnarray*}
In view of Lemma \ref{LemContEmbed} we obtain that $(u_n-u_0+tu_0)_{n\in N}$ is bounded in $L^\mu(\R^N)$ and in $L^p(\R^N)$. Therefore
for any $\eps>0$ there is $\delta>0$
such that for any $\Omega$ with the Lebesgue measure $|\Omega|<\delta$, we have
$$\int_{\Omega}|g(x,u_n-u_0+tu_0)u_0|\; dx < \eps$$
for any $n\in\N$.
Thus $(g(x,u_n-u_0+tu_0)u_0)_{n\in\N}$ is uniformly integrable.
Moreover for any $\eps>0$ there is $\Omega\subset\R^N$,
$|\Omega|<+\infty$,
such that for any $n\in\N$
$$\int_{\R^N\setminus\Omega}|g(x,u_n-u_0+tu_0)u_0|\; dx < \eps.$$
Hence a family $(g(x,u_n-u_0+tu_0)u_0)_{n\in\N}$ is tight over $\R^N$.
Since $g(u_n-u_0+tu_0)u_0\to g(tu_0)u_0$ a.e. in $\R^N$, then
in view of the Vitali convergence theorem $g(x,tu_0)u_0$ is integrable and
$$\int_{\R^N}g(x,u_n-u_0+tu_0)u_0\; dx\to \int_{\R^N}g(x,tu_0)u_0\; dx$$
as $n\to\infty$.
By (\ref{EqBrezisLieb}) we obtain
$$\int_{\R^N}G(x,u_n)-G(x,u_n-u_0)\; dx\to
\int_0^1\int_{\R^N}g(x,tu_0)u_0\; dxdt=\int_{\R^N}G(x,u_0)\; dx$$
as $n\to\infty$.
Taking into account (\ref{EqConvG}) we get
$$\lim_{n\to\infty}\int_{\R^N}G(x,u_n-u_0)\; dx=0$$
and by (\ref{eqG3estimate}) we have $u_n\to u_0$ in $L^{2,\mu}(\R^N)$. Hence $u_n'\to u_0'$ in $L^{2,\mu}(\R^N)$.\\
$(c)$ 
Let $u\in\mathcal{N}$. Note that for any $t\geq 0$ and $h'\in E'_{2,\mu}$ 
$$\J(t u+h')-\J(u)=\frac{t^2-1}{2}\|u^+\|^2
+I(u)-I(tu+h').$$
Since $u\in\mathcal{N}$
and $\J'(u)(u)=\|u^+\| ^2-I'(u)(u)$, then for $u\neq tu+h'$ 
\begin{eqnarray*}
\J(t u+h')-\J(u)&=&
I'(u)\Big(\frac{t^2-1}{2}u+th'\Big)
+I(u)-I(tu+h')\\
&=&-\frac12\|h'\|_E^2+\int_{\R^N}g(x,u)\Big(\frac{t^2-1}{2}u+th'\Big)
+G(x,u)-G(x,tu+h')\;dx,\\
&<&0
\end{eqnarray*}
where the last inequality follows from \cite{SzulkinWeth}[Lem. 2.2].
\end{proof}

Since $E^+\subset H^1(\R^N)$ and $H^1(\R^N)$ is not compactly embedded in $L^{\mu}(\R^N)$ and $L^{p}(\R^N)$, then we do not know if $\J$ satisfies $(PS)_c^{\cT}$-condition in $\mathcal{N}$ (see Section \ref{SectionAbstactSetting}, cf. \cite{BartschMederski}). Moreover Palais-Smale sequences do not have to be bounded since we do not assume (\ref{ARcond}) (cf. \cite{Jeanjean}). However the boundedness is attainable on $\mathcal{N}$. 

\begin{Lem}\label{LemCoercive}
$\J$ is coercive on $\mathcal{N}$, i.e. $\J(u)\to\infty$ as $\|u\|\to\infty$, $u\in\mathcal{N}$.
\end{Lem}
\begin{proof}
Suppose that $\|u_n\|\to\infty$ as $n\to\infty$, $u_n\in\mathcal{N}$ and 
$\J(u_n)\leq c_1$ for some constant $c_1>0$. Let $v_n:=\frac{u_n}{\|u_n\|}$. Since $E_{2,\mu}$ is reflexive and compactly embedded in $L_{loc}^{2}(\R^N)$ then, up to a subsequence, $v_n\rightharpoonup v$ in $E_{2,\mu}$ and $v_n(x)\to v(x)$ a.e. in $\R^N$. Moreover there is a sequence $(y_n)_{n\in\N}\subset\R^N$ such that
\begin{equation}\label{EqLemLions}
\liminf_{n\to\infty}\int_{B(y_n,1)}|v_n^+|^2\; dx >0.
\end{equation}
Otherwise, in view of Lions lemma (see \cite{Willem}[Lemma 1.21]) we get that $v_n^+\to 0$ in $L^t(\R^N)$ for $2<t<2^*$. By $(G2)$ we get $\int_{\R^N}G(x,sv_n^+)\;dx\to 0$ for any $s\geq0$. Let us fix $s\geq0$. Hence by Lemma \ref{LemLG} $(c)$
\begin{equation}\label{EqIneq2}
c_1\geq \limsup_{n\to\infty}\J(u_n)\geq \limsup_{n\to\infty} \J(sv_n^+)= \frac{s^2}{2}\limsup_{n\to\infty}\|v_n^+\|^2.
\end{equation}
By (\ref{eqG3estimate}) and in view of Theorem \ref{ThLink1} $(a)$ we have
$$\frac{1}{2}(\|u_n^+\|^2-\|u_n'\|_E^2)-c\min\{|u_n|_{2,\mu}^2,|u_n|_{2,\mu}^\mu\}\geq \J(u_n)\geq c_{\inf}:=\inf_{\mathcal{N}}\J>0.$$
If $\liminf_{n\to\infty}|u_n|_{2,\mu}=0$ then, up to a subsequence, $|u_n|_{2,\mu}\to 0$, and for sufficiently large $n$
\begin{eqnarray*}
2\|u_n^+\|^2&\geq& \|u_n^+\|^2+\|u_n'\|_E^2+2c_{\inf}+2c\min\{|u_n|_{2,\mu}^2,|u_n|_{2,\mu}^\mu\}\\
&\geq& \|u_n^+\|^2+\|u_n'\|_E^2+|u_n|_{2,\mu}^2=\|u_n\|^2.
\end{eqnarray*}
If $\liminf_{n\to\infty}|u_n|_{\mu}>0$ then there is  $c_2\in (0,1)$ 
such that for sufficiently large $n$
\begin{eqnarray*}
2\|u_n^+\|^2&\geq& \|u_n^+\|^2+\|u_n'\|_E^2+2c_{\inf}+2c\min\{|u_n|_{2,\mu}^2,|u_n|_{2,\mu}^\mu\}\\
&\geq& c_2(\|u_n^+\|^2+\|u_n'\|_E^2+|u_n|_{2,\mu}^2)=c_2\|u_n\|^2.
\end{eqnarray*}
Therefore, passing to a subsequence if necessary, $c_3:=\inf_{n\in\N}\|v_n^+\|^2>0$
and by (\ref{EqIneq2})
$$c_1\geq \frac{s^2}{2}c_3$$
for any $s\geq0$. The obtained contradiction shows that (\ref{EqLemLions}) holds. Then we may assume that $(y_n)\in\Z^N$ and
$$\liminf_{n\to\infty}\int_{B(y_n,r)}|v_n^+|^2\; dx >0$$
for some $r>1$. Since $\J$ and $\mathcal{N}$ are invariant under translations of the form $u\mapsto u(\cdot-k)$, $k\in\Z^N$, then we may assume that $v_n^+\to v^+$ in $L^2_{loc}(\R^N)$ and $v^+\neq 0$. Note that if $v(x)\neq 0$ then $u_n(x)=v_n(x)\|u_n\|\to\infty$
and by  $(G4)$
$$\frac{G(x,u_n(x))}{\|u_n\|^2}=\frac{G(x,u_n(x))}{|u_n(x)|^2}|v_n(x)|^2\to\infty$$
as $n\to\infty$.
Therefore  by Fatou's lemma
\begin{eqnarray*}
\frac{\J(u_n)}{\|u_n\|^2}&=&\frac{1}{2}(\|v_n^+\|^2-\|v_n'\|^2_E)
-\int_{\R^N}\frac{G(x,u_n(x))}{\|u_n\|^2}\;dx\\
&\to&-\infty.
\end{eqnarray*}
Thus we get a contradiction and we conclude the coercivity.
\end{proof}

\noindent{\it Proof of Theorem \ref{Th1}.} In view of Theorem \ref{ThLink1} $(a)$
$$c_{\inf}=\inf_{\mathcal{N}}\J>0$$
and
there exists a $(PS)_{c_{\inf}}$-sequence $(u_n)_{n\in \N}\subset\mathcal{N}$, i.e. $\J(u_n)\to c_{\inf}$ and $\J'(u_n)\to 0$ as $n\to\infty$. By Lemma \ref{LemCoercive} we get that $(u_n)_{n\in \N}$ is bounded and after passing to a subsequence $u_n\rightharpoonup u$ in $E_{2,\mu}$. 
Then there is a sequence $(y_n)\in\R^N$ such that
\begin{equation}\label{EqLemLions1}
\liminf_{n\to\infty}\int_{B(y_n,1)}|u_n^+|^2\; dx >0.
\end{equation}
Otherwise, in view of Lions lemma (see \cite{Willem}[Lemma 1.21]), $u_n^+\to 0$ in $L^t(\R^N)$ for $2<t<2^*$. By $(G2)$ we obtain
$$\|u_n^+\|^2=\J'(u_n)(u_n^+)+\int_{\R^N}g(x,u_n)u_n^+\; dx\to 0$$
as $n\to\infty$. Hence
$$0<c_{\inf}=\lim_{n\to\infty}\J(u_n)\leq \lim_{n\to\infty}\frac{1}{2}\|u_n^+\|^2=0$$
and we get a contradiction. Therefore (\ref{EqLemLions1}) holds 
and we may assume that there is a sequence $(y_n)\in\Z^N$ such that
\begin{equation}\label{EqLemLions2}
\liminf_{n\to\infty}\int_{B(y_n,r)}|u_n^+|^2\; dx >0
\end{equation}
for some $r>1$. 
Since $\|u_n(\cdot + y_n)\|=\|u_n\|$, then there is $u\in E_{2,\mu}$ such that, up to a subsequence, $u_n(\cdot + y_n)\rightharpoonup u$ in $E_{2,\mu}$, $u_n(x+y_n)\to u(x)$ a.e. on $\R^N$ and $u_n^+(\cdot + y_n)\to u^+$ in $L^2_{loc}(\R^N)$.
By (\ref{EqLemLions2}) we get $u^+\neq 0$ and then $u\neq 0$.
Since $\J$ and $\mathcal{N}$ are invariant under translations of the form $u\mapsto u(\cdot+y)$, $y\in\Z^N$, then $\J'(u)=0$. Observe that 
$u\in \mathcal{N}$, and by $(G2)$ and $(G5)$
$$\frac{1}{2} g(x,u_n(x+y_n))u_n(x+y_n)-
G(x,u_n(x+y_n))\geq 0.$$
Therefore, in view of the Fatou's lemma
$$c_{\inf}=\lim_{n\to\infty}\J(u_n(\cdot+y_n)) =
 \lim_{n\to\infty}
\Big(\mathcal{J}(u_n(\cdot+y_n)) -
\frac{1}{2}\mathcal{J}'(u_n(\cdot+y_n))u_n(\cdot+y_n)\Big) 
\geq \J(u).$$
Thus we get $\J(u)=c_{\inf}$. Since $u\in E_{2,\mu}$ is a  solution to (\ref{NSE}), then by Corollary \ref{CorollaryUVanish} we get $u(x)\to 0$ as $|x|\to\infty$.
\hfill $\square$\\

\section{Multiple solutions}\label{SectionMultiplicity}

Note that if $u\in E_{2,\mu}$ is a critical point of $\J$ then the orbit under the action of $\Z^N$, $\mathcal{O}(u):=\{u(\cdot -k)|\;k\in\Z^N\}$
consists of critical points. Two critical points $u_1,u_2\in E_{2,\mu}$ are said to be {\em geometrically distinct} if $\mathcal{O}(u_1)\cap \mathcal{O}(u_2)=\emptyset$. 
In view of Theorem \ref{ThLink1} $(b)$ we know that $\Psi:=\J\circ n: S^+\to \R$ is a $C^1$ map. Observe that in order to prove Theorem \ref{Th2} it is enough to show that $\Psi$ has infinitely many geometrically distinct critical points (see Theorem \ref{ThLink1} $(b)$). The following lemma is crucial in the consideration of the multiplicity of critical points (cf. \cite{SzulkinWeth}[Lem. 2.14]).

\begin{Lem}\label{LemDiscretePS}
Let $d\geq c_{\inf}$. If $(u_n^1)$, $(u_n^2)\subset \Psi^d:=\{u\in S^+|\; \Psi(u)\leq d\}$ are two Palais-Smale sequences for $\Psi$, then either $\|u_n^1-u_n^2\|\to 0$ as $n\to\infty$ or 
\begin{equation}\label{EqDiscretePS}
\limsup_{n\to\infty}\|u_n^1-u_n^2\|\geq \rho(d)\inf\{\|u_1-u_2\||\; \Psi'(u_1)=\Psi'(u_2)=0,\; u_1\neq u_2\in S^+\}, 
\end{equation}
where $\rho(d)>0$ depends on $d$ but not on the particular choice of Palais-Smale sequences.
\end{Lem}
\begin{proof}
Let $(u_n^1)$, $(u_n^2)\subset \Psi^d:=\{u\in S^+|\; \Psi(u)\leq d\}$ be two Palais-Smale sequences for $\Psi$. Let us consider two cases.\\
\noindent {\em Case 1:} $|n(u_n^1)^+-n(u_n^2)^+|_{\mu}\to 0$ and $|n(u_n^1)^+-n(u_n^2)^+|_{p}\to 0$.
Observe that by $(G2)$
\begin{eqnarray*}
\|n(u_n^1)^+-n(u_n^2)^+\|^2&=&\J'(n(u_n^1))(n(u_n^1)^+-n(u_n^2)^+)
-\J'(n(u_n^2))(n(u_n^1)^+-n(u_n^2)^+)\\
&&+ \int_{\R^N} (g(x,n(u_n^1))-g(x,n(u_n^2)))(n(u_n^1)^+-n(u_n^2)^+)\; dx\\
&\leq & \J'(n(u_n^1))(n(u_n^1)^+-n(u_n^2)^+)
-\J'(n(u_n^2))(n(u_n^1)^+-n(u_n^2)^+)\\
&& + a (|n(u_n^1)|_\mu^{\mu-1}+
|n(u_n^2)|_\mu^{\mu-1})\cdot |n(u_n^1)^+-n(u_n^2)^+|_\mu\\
&&  + a (|n(u_n^1)|_p^{p-1}+
|n(u_n^2)|_p^{p-1})\cdot |n(u_n^1)^+-n(u_n^2)^+|_p.
\end{eqnarray*}
By Theorem \ref{ThLink1} $(b)$ we know that $(n(u_n^1))_{n\in\N}$ and $(n(u_n^2))_{n\in\N}$ are Palais-Smale sequences for $\J$ and, by Lemma \ref{LemCoercive}, they are bounded in $E_{2,\mu}$. Since $E_{2,\mu}$ is continuously embedded in $L^\mu(\R^N)$ and in $L^p(\R^N)$, then
$$\|n(u_n^1)^+-n(u_n^2)^+\|\to 0.$$
Observe that, if $u=u^++u'\in \mathcal{N}$ then inequality $\J(u)\geq c_{\inf}$ implies that
\begin{equation}
\|u^+\|\geq \max\{ \sqrt{2c_{\inf}}, \|u'\|_E \}.
\end{equation}
Therefore similarly as in \cite{SzulkinWeth}[Lem. 2.13] we infer that
$$\|u_n^1-u_n^2\|=\Big\|\frac{n(u_n^1)^+}{\|n(u_n^1)^+\|}-\frac{n(u_n^1)^+}{\|n(u_n^1)^+\|}\Big\|\leq 
\sqrt{\frac{2}{c_{\inf}}}\|n(u_n^1)^+-n(u_n^2)^+\|.$$
Thus
$$\|u_n^1-u_n^2\|\to 0.$$

\noindent {\em Case 2:} $|n(u_n^1)^+-n(u_n^2)^+|_{\mu}\nrightarrow 0$ or $|n(u_n^1)^+-n(u_n^2)^+|_{p}\nrightarrow 0$.\\
In view of Lions lemma \cite{Willem}[Lemma 1.21]
there is a sequence $(y_n)\in\Z^N$ and $r>1$ such that
\begin{equation}\label{EqLemLions3}
\liminf_{n\to\infty}\int_{B(y_n,r)}|n(u_n^1)^+-n(u_n^2)^+|^2\; dx >0.
\end{equation}
Then we may assume that, up to a subsequence, 
$$n(u_n^1)(\cdot+y_n)\rightharpoonup v_1,\; n(u_n^2)(\cdot+y_n)\rightharpoonup v_2 \hbox{ in }E_{2,\mu},$$
$$n(u_n^1)^+(\cdot+y_n)\to v_1^+,\; n(u_n^2)^+(\cdot+y_n)\to v_2^+ \hbox{ in }L^2_{loc}(\R^N)$$
and
$$\|n(u_n^1)^+(\cdot + y_n)\|\to \alpha_1,\; \|n(u_n^2)^+(\cdot + y_n)\|\to \alpha_2$$
for some $\alpha_1$, $\alpha_2\geq \sqrt{2c_{\inf}}$.
From (\ref{EqLemLions3}) we infer that $v_1^+\neq v_2^+$ and thus $v_1\neq v_2$.
Since $n$, $n^{-1}$, $\J'$, $(\J\circ n)'$ are equivariant with respect to $\Z^N$-action, then $\J'(v_1)=\J'(v_2)=0$. Observe that if $v_1\neq 0$ and $v_2\neq 0$ then $v_1$, $v_2\in\mathcal{N}$ and
\begin{eqnarray*}
\liminf_{n\to\infty}\|u_n^1-u_n^2\|&=&
\liminf_{n\to\infty}\|(u_n^1-u_n^2)(\cdot + y_n)\|\\
&=&
\liminf_{n\to\infty}\Big\|\frac{n(u_n^1)^+(\cdot + y_n)}{\|n(u_n^1)^+(\cdot + y_n)\|}-\frac{n(u_n^2)^+(\cdot + y_n)}{\|n(u_n^2)^+(\cdot + y_n)\|}\Big\|\\
&\geq& \Big\|\frac{v_1^+}{\alpha_1}-\frac{v_2^+}{\alpha_2}\Big\|
\geq \min\Big\{\frac{\|v_1^+\|}{\alpha_1},\frac{\|v_2^+\|}{\alpha_2}\Big\}
\Big\|\frac{v_1^+}{\|v_1^+\|}-\frac{v_2^+}{\|v_2^+\|}\Big\|\\
&\geq& \frac{\sqrt{2c_{\inf}}}{s(d)}
\Big\|n^{-1}(v_1)-n^{-1}(v_2)\Big\|
\end{eqnarray*}
where $s(d):=\sup\{\|u^+\||\; u\in\mathcal{N},\;\J(u)\leq d\}$. 
By Theorem \ref{ThLink1}, $n^{-1}(v_1)$, $n^{-1}(v_2)$ are critical points of $\Psi$ and  we get (\ref{EqDiscretePS}).
Note that if $v_1=0$ or $v_2=0$ then, similarly as above, we show that
$$\liminf_{n\to\infty}\|u_n^1-u_n^2\|\geq \frac{\sqrt{2c_{\inf}}}{s(d)}.$$
and again (\ref{EqDiscretePS}) holds.
\end{proof}

\noindent{\it Proof of Theorem \ref{Th2}.} 
Let $g$ be odd. In view of Theorem \ref{ThLink1} $(b)$ we get that $n$ is equivariant with respect to the $\Z^N$-action given by $u\mapsto u(\cdot-k)$ for $k\in \Z^N$. Moreover $\J$ is even and $n$ is odd. Therefore $\Psi$ is even and invariant with respect to the $\Z^N$-action.
Let $\F$ be the set of geometrically distinct critical points of $\Psi$ and assume that $\F$  is finite. Then, 
similarly as in \cite{SzulkinWeth}[Lem. 2.13], we show that  
$$\inf\{\|u_1-u_2\||\; \Psi'(u_1)=\Psi'(u_2)=0,\; u_1\neq u_2\in S^+\}>0.$$
The obtained discreteness of Palais-Smale sequences in Lemma \ref{LemDiscretePS} allows us to  
repeat the following arguments: Lemma 2.15, Lemma 2.16 and proof of Theorem 1.2 from \cite{SzulkinWeth}. In fact, we show that for any $k\in\N$ there is $u\in S^+$ such that $\Psi'(u)=0$ and $\Psi(u)=c_k$, where
$$c_k:=\inf\{d\in\R|\; \gamma(\Psi^d)\geq k\}$$
and $\gamma$ denotes the usual Krasnoselskii genus (see \cite{Struwe}). Moreover $c_k<c_{k+1}$ for any $k\in\Z$ and thus we get the contradiction (see \cite{SzulkinWeth} for detailed arguments).
In view of Theorem \ref{ThLink1} $(b)$ we obtain the existence of infinitely many geometrically distinct solutions to (\ref{NSE}).
\hfill $\square$\\

{\bf Acknowledgment.}
The author would like to thank the referee for the valuable comments and suggestions helping to improve the paper.

\noindent {\sc Address of the author:}\\[1em]
\parbox{8cm}{Jaros\l aw Mederski\\
 Nicolaus Copernicus University \\
 ul.\ Chopina 12/18\\
 87-100 Toru\'n\\
 Poland\\
 jmederski@mat.umk.pl\\
 }

\begin{thebibliography}{99}
\baselineskip 2 mm 
\bibitem{AlamaLi} S. Alama, Y.Y. Li,  {\em On ''multibump'' bound states for certain semilinear elliptic equations}, Indiana Univ. Math. J. {\bf 41} (1992), no. 4, 983--1026. 
\bibitem{AlvesSoutoMontenegro} C.O. Alves, M.A.S. Souto, M. Montenegro, {\em Existence of solution for two classes of elliptic problems in $\R^N$ with zero mass}, J. Differential Equations {\bf 252} (2012),  5735--5750.
\bibitem{BadialePisaniRolando} M. Badiale, L. Pisani, S. Rolando, {\em Sum of wheighted Lebesgue spaces and nonlinear elliptic equations}, Nonlinear Differential Equations and Applications {\bf 18} (2011), 369--405.
\bibitem{BartschDingPeriodic} T. Bartsch, Y. Ding, {\em On a nonlinear Schr\"odinger equation with periodic potential}, Math. Ann. {\bf 313} (1999), no. 1, 15--37. 
\bibitem{BartschMederski} T. Bartsch, J. Mederski, {\em Ground and bound state solutions of semilinear time-harmonic Maxwell equations in a bounded domain}, arXiv:1310.4731.
\bibitem{BerLions} H. Berestycki, P.L. Lions, {\em Nonlinear scalar field equations. I. Existence of a ground state}, Arch. Rational Mech. Anal. {\bf 82}, (1983), 313--345.
\bibitem{BuffoniJeanStuart} B. Buffoni, L. Jeanjean, C. A. Stuart, {\em Existence of a nontrivial solution to a strongly indefinite semilinear equation}, Proc. Amer. Math. Soc. {\bf 119} (1993), no. 1, 179--186. 
\bibitem{CotiZelati} V. Coti-Zelati, P. Rabinowitz, {\em Homoclinic type solutions for a semilinear elliptic PDE on $\R^n$}, Comm. Pure Appl. Math. {\bf 45} (1992), no. 10, 1217--1269.
\bibitem{DingLee} Y. Ding, C. Lee, {\em Multiple solutions of Schr\"odinger equations with indefinite linear part and super or asymptotically linear terms}, J. Differential Equations {\bf 222} (2006), no. 1, 137--163. 
\bibitem{EvequozWeth} G. Evequoz, T. Weth, {\em Real solutions to the nonlinear Helmholtz equation with local nonlinearity}, Arch. Rat. Mech. Anal. {\bf 211} (2014), 359--388.
\bibitem{GilbargTrudinger} D. Gilbarg, N.S. Trudinger, {\em Elliptic partial differential equations of second order}, Springer-Verlag, Berlin, 2001.
\bibitem{Jeanjean} L. Jeanjean, {\em On the existence of bounded Palais-Smale sequences and application to a Landesman-Lazer-type problem set on $\R^N$},  Proc. Roy. Soc. Edinburgh Sect. A {\bf 129} (1999), no. 4, 787--809. 
\bibitem{KryszSzulkin} W. Kryszewski, A. Szulkin, {\em Generalized linking theorem with an application to semilinear Schr\"odinger equation}, Adv. Diff. Eq. {\bf 3} (1998), 441--472.
\bibitem{LiWangZeng} Y. Li, Z.-Q. Wang, J. Zeng, {\em Ground states of nonlinear Schr\"odinger equations with potentials}, Ann. Inst. H. Poincaré Anal. Non Linéaire {\bf 23} (2006), no. 6, 829--837. 
\bibitem{Liu} S. Liu, {\em On superlinear Schr\"odinger equations with periodic potential}, Calc. Var. Partial Differential Equations {\bf 45} (2012), no. 1-2, 1--9.
\bibitem{LiuWang} Z. Liu, Z.-Q. Wang, {\em On the Ambrosetti–Rabinowitz superlinear condition}, Adv. Nonlinear Stud. {\bf 4} (2004), 561--572.
\bibitem{MiyagakiSouto} O. H. Miyagaki, M. A. S. Souto, 
{\em Superlinear problems without Ambrosetti and Rabinowitz growth condition}, 
J. Differential Equations {\bf 245} (2008), no. 12, 3628--3638. 
\bibitem{Pankov} A. Pankov, {\em Periodic Nonlinear Schr\"odinger Equation with Application to Photonic Crystals}, Milan J. Math. {\bf 73} (2005), 259--287.
\bibitem{PankovNotes} A. Pankov, {\em Lecture Notes on Schrodinger Equations}, Nova Science Pub Inc. 2008.
\bibitem{Rabinowitz} P. Rabinowitz, {\em On a class of nonlinear Schr\"odinger equations}, Z. Angew. Math. Phys. {\bf 43} (1992), 270--291.
\bibitem{ReedSimon} M. Reed, B. Simon, {\em Methods of Modern Mathematical Physics, Analysis of Operators, Vol. IV}, Academic Press, New York, 1978.
\bibitem{Struwe} M. Struwe, {\em Variational Methods}, Springer 2008.
\bibitem{SchechterZou} M. Schechter, W. Zou, {\em Weak linking theorems and Schr\"odinger equations with critical Sobolev exponent}, ESAIM Control Optim. Calc. Var. {\bf 9} (2003), 601--619.
\bibitem{SzulkinWeth} A. Szulkin, T. Weth, {\em Ground state solutions for some indefinite variational problems}, J. Funct. Anal. {\bf 257} (2009), no. 12, 3802--3822. 
\bibitem{SzulkinWethHandbook} A. Szulkin, T. Weth, {\em The method of Nehari manifold. Handbook of nonconvex analysis and applications}, Handbook of nonconvex analysis and applications, 597--632, Int. Press, Somerville, 2010.
\bibitem{TroestlerWillem} C. Troestler, M. Willem, {\em Nontrivial solution of a semilinear Schr\"odinger equation}, Comm. Partial Differential Equations {\bf 21} (1996), 1431--1449.
\bibitem{Willem} M. Willem, {\em Minimax Theorems}, Birkhäuser Verlag 1996.
\bibitem{WillemZou} M. Willem, W. Zou, {\em On a Schr\"odinger equation with periodic potential and spectrum point zero}, Indiana Univ. Math. J. {\bf 52} (2003), no. 1, 109--132.
\bibitem{YangChenDing} M. Yang, W. Chen, Y. Ding, {\em Solutions for periodic Schr\"odinger equation with spectrum zero and general superlinear nonlinearities}, J. Math. Anal. Appl. {\bf 364} (2010), no. 2, 404--413.

\end{thebibliography}
\end{document}